\documentclass[11pt]{amsart}

\usepackage[utf8]{inputenc}

\usepackage{color}
\usepackage{enumerate}

\usepackage[all]{xy}

\usepackage{amsaddr}
\usepackage{amssymb}
\usepackage{a4wide}
\usepackage{fixme}

\usepackage{marginnote}

\newcommand{\Z}{\mathbb{Z}}

\usepackage{xargs}                      
\usepackage[pdftex,dvipsnames]{xcolor}  

\usepackage{hyperref}
\hypersetup{final} 

\title[Nonunital prime rings graded by ordered groups]{Nonunital prime rings graded by ordered groups}

\newtheorem{thm}{Theorem}
\newtheorem{prop}[thm]{Proposition}
\newtheorem{lem}[thm]{Lemma}
\newtheorem{cor}[thm]{Corollary}
\theoremstyle{definition}
\newtheorem{defi}[thm]{Definition}
\newtheorem{exa}[thm]{Example}
\newtheorem{rem}[thm]{Remark}

\newcommand{\Q}{\mathbb{Q}}

\author{Daniel L\"{a}nnstr\"{o}m}
\address{Advanced Programs, Aeronautics, 
SAAB AB, SE-58133 Link\"{o}ping, Sweden}
\author{Patrik Lundstr\"{o}m}
\address{Department of Engineering Science,
University West, SE-46186 Trollh\"{a}ttan, Sweden}
\author{Johan \"{O}inert}
\address{Department of Mathematics and Natural Sciences,
Blekinge Institute of Technology,
SE-37179 Karlskrona, Sweden}
\address{Department of Engineering,
University of Sk\"{o}vde,
SE-54128 Sk\"{o}vde, Sweden}
\author{Stefan Wagner}
\address{Department of Mathematics and Natural Sciences,
Blekinge Institute of Technology,
SE-37179 Karlskrona, Sweden}

\email{daniel.lannstrom@saabgroup.com}
\email{patrik.lundstrom@hv.se}
\email{johan.oinert@bth.se}
\email{stefan.wagner@bth.se}

\makeatletter
\@namedef{subjclassname@2020}{%
  \textup{2020} Mathematics Subject Classification}
\makeatother



\date{\today}

\begin{document}

\begin{abstract}
Let \( G \) be a group with identity element \( e \), and suppose that \( S \) is an associative \( G \)-graded ring that is not necessarily unital.
In the case where $G$ is an ordered group, we show that a graded ideal is prime if and only if it is graded prime. Consequently, in that setting, a graded ring is prime if and only if it is graded prime.
For any group \( G \), if \( S \) is what we call \emph{ideally symmetrically \( G \)-graded}, then we show that there is a bijective correspondence between the \( G \)-graded prime ideals of \( S \) and the \( G \)-prime ideals of \( S_e \). We use this correspondence in the case where \( G \) is ordered and \( S \) is ideally symmetrically \( G \)-graded to show that \( S \) is prime if and only if \( S_e \) is \( G \)-prime.
These results generalize classical theorems by N\u{a}st\u{a}sescu and Van Oystaeyen to a nonunital setting. As applications, we provide a new proof of a primeness criterion for Leavitt path rings and establish conditions for primeness of symmetrically $G$-graded subrings of group rings over fully idempotent rings.
\end{abstract}

\maketitle


\section{Introduction}\label{sec:introduction}

Throughout this article, \( S \) denotes an associative but not necessarily unital ring. Unless explicitly stated, all ideals of \( S \) are assumed to be two-sided.
Recall that a proper ideal $P$ of $S$ is called \emph{prime} if, for all ideals \( A, B \subseteq S \), the inclusion \( AB \subseteq P \) implies \( A \subseteq P \) or \( B \subseteq P \). The ring \( S \) is called \emph{prime} if the zero ideal \( \{0\} \) is prime. This notion generalizes that of integral domains to the noncommutative setting. In fact, a commutative ring is prime if and only if it is an integral domain. The class of prime rings includes many familiar constructions, such as simple rings, primitive rings (left or right), and matrix rings over integral domains.

Let \( G \) be a group with identity element \( e \). Recall that \( S \) is said to be \emph{\( G \)-graded} if there are additive subgroups \( S_x \subseteq S \), for \( x \in G \), such that $S = \oplus_{x \in G} S_x$ and $S_x S_y \subseteq S_{xy}$, for all $x,y \in G$.
If, additionally, \( S_x S_y = S_{xy} \) for all \( x, y \in G \), then \( S \) is said to be \emph{strongly \( G \)-graded}. 
Throughout the rest of this introduction, let $S$ be a $G$-graded ring.

A fundamental problem, studied over the past five decades, is to determine when a graded ring is prime; see e.g. \cite{abrams1993,cohen1983,connell1963,lannstrom2025,nastasescu1982,nastasescu2004,passman1962,passman1970,passman1977,passman1983,passman1984,passmaninfinite1984}. In the unital and strongly graded case, Passman provided a complete solution (see Theorem~\ref{thm:main}), relying on the following notions. Let \( H, N \) be subgroups of \( G \).
Put \( S_N := \oplus_{x \in N} S_x \).
For an ideal $I$ of $S_N$ and \( x \in G \), define \( I^x := S_{x^{-1}} I S_x \).  The ideal \( I \) is said to be \emph{\( H \)-invariant} if \( I^x \subseteq I \) for all \( x \in H \). 

\begin{thm}[Passman {\cite[Thm.~1.3]{passmaninfinite1984}}]\label{thm:main}
Suppose that $S$ is a unital and strongly $G$-graded ring.
Then $S$ is not prime if and only if there exist:
\begin{enumerate}[{\rm (i)}]
    \item subgroups $N \lhd H \subseteq G$ with $N$ finite,
    \item an $H$-invariant ideal $I$ of $S_e$ such that
    $I^x I = \{ 0 \}$ for every $x \in G \setminus H$, and
    \item nonzero $H$-invariant 
    ideals $\tilde{A}, \tilde{B}$ of $S_N$ such that
    $\tilde{A},\tilde{B} \subseteq I S_N$ and
    $\tilde{A}\tilde{B} = \{ 0 \}$.
\end{enumerate}
\end{thm}

In \cite[Thm.~1.3]{lannstrom2025}, the authors of 
the present article extended Theorem~\ref{thm:main} to a class of nonunital rings called \emph{nearly epsilon-strongly $G$-graded rings} (see \cite{nystedtoinert2019}). A $G$-graded ring \( S \) belongs to this class if, for each \( x \in G \) and each \( s \in S_x \), there exist \( \epsilon_x \in S_x S_{x^{-1}} \) and \( \epsilon_x' \in S_{x^{-1}} S_x \) such that
$\epsilon_x s = s = s \epsilon_x'$.

In contrast to the intricate conditions in Theorem~\ref{thm:main}, N\u{a}st\u{a}sescu 
and Van Oystaeyen have shown (see Theorem~\ref{thm:mainordered}) that primeness criteria for graded rings become much simpler if we only consider groups \( G \) that are \emph{ordered}, 
meaning that
$G$ is equipped with a total order \( \leq \) satisfying \( a \leq b \Rightarrow xay \leq xby \) for all \( a,b,x,y \in G \). 
Before stating their results,
we recall the following notions.
An ideal $I$ of $S$ is \emph{graded} if \( I = \oplus_{x \in G} I_x \), where \( I_x := I \cap S_x \). A proper graded ideal \( P \) of $S$ is \emph{graded prime} if \( AB \subseteq P \) implies \( A \subseteq P \) or \( B \subseteq P \) for all graded ideals \( A, B \subseteq S \). The ring \( S \) is \emph{graded prime} if \( \{0\} \) is a graded prime ideal. A proper \( G \)-invariant ideal \( Q \subseteq S_e \) is \emph{\( G \)-prime} if \( AB \subseteq Q \) implies \( A \subseteq Q \) or \( B \subseteq Q \) for all \( G \)-invariant ideals \( A, B \subseteq S_e \). The ring \( S_e \) is \emph{\( G \)-prime} if \( \{0\} \) is a \( G \)-prime ideal. 

\begin{thm}[N\u{a}st\u{a}sescu and Van Oystaeyen {\cite[Prop.~II.1.4]{nastasescu1982}} and {\cite[Prop.~2.11.7]{nastasescu2004}}]\label{thm:mainordered}
Suppose that $S$ is a unital and $G$-graded ring 
where $G$ is ordered. The following assertions hold:
\begin{enumerate}[{\rm (a)}]

\item The ring $S$ is prime if and only if it is graded prime.

\item Suppose that $S$ is strongly $G$-graded. Then $S$ is prime if and only if $S_e$ is $G$-prime.

\end{enumerate}
\end{thm}

In this article, we continue the investigation from \cite{lannstrom2025}, now focusing on nonunital analogues of Theorem~\ref{thm:mainordered}.
We prove that Theorem~\ref{thm:mainordered}(a) holds 
for \emph{all} nonunital rings (see Theorem~\ref{thm:Gorderedprimegradedprime}).
We also extend Theorem~\ref{thm:mainordered}(b) to rings \( S \) that are \emph{ideally symmetrically $G$-graded}, meaning that for every graded ideal $I$ of $S$ and every \( x \in G \), the equalities \( S_x S_{x^{-1}} I_x = I_x \) and \( I_x S_{x^{-1}} S_x = I_x \) hold (see Theorem~\ref{thm:GorderedGprime}).
This class strictly contains all nearly epsilon-strongly $G$-graded rings, and thus also all unital strongly $G$-graded rings (see Proposition~\ref{prop:RL}). 
Therefore, the results presented here apply 
 to a broader range of gradings than those considered in 
\cite{lannstrom2025}.

Here is a detailed outline of this article.
In Section~\ref{sec:gradedprimeideals}, we extend some
results on prime ideals to a nonunital setting. 
We introduce \emph{graded $m$-systems} to establish a primeness criterion
for a nonunital $S$ 
(see
Corollary~\ref{cor:msystem}
and
Theorem~\ref{thm:Pgradedprime}).
We also prove a nonunital 
version of Theorem~\ref{thm:mainordered}(a) (see 
Theorem~\ref{thm:Gorderedprimegradedprime}).
In Section~\ref{sec:Gprimeideals}, we
establish a one-to-one correspondence between 
$G$-invariant ideals of $S_e$ and graded ideals of $S$, in the case where $S$ is ideally symmetrically graded 
(see Proposition~\ref{prop:leftright}).
We also generalize 
Theorem~\ref{thm:mainordered}(b) to the class of ideally symmetrically graded rings
(see Theorem~\ref{thm:GorderedGprime}).
In Section~\ref{sec:LPA},
we use our results from Section~\ref{sec:Gprimeideals} 
to give a new short proof (see
Theorem~\ref{thm:shortLPA}) of a previously established  characterization of prime Leavitt path rings 
from \cite{lannstrom2025}. 
In Section~\ref{sec:primegradedsubringsofgrouprings},
we investigate primeness of graded subrings of 
group rings $R[G]$ for nonunital rings $R$.
In particular, we show that for any nontrivial 
ordered group $G$,
there are prime $G$-graded rings which are not 
nearly-epsilon strongly $G$-graded (see
Proposition~\ref{prop:RL}).

\section{Graded prime ideals}\label{sec:gradedprimeideals}

Throughout this section, \( S \) denotes an associative but not necessarily unital ring. 
Let $A$ be a subset of $S$. 
We let $\mathbb{Z}A$ denote the set of all finite sums of integer multiples of elements of $A$, and we let $SA$ denote the set of all finite sums of products of elements of $S$ with elements of $A$.
The sets $AS$ and $SAS$ are defined analogously.
The ideal of $S$ 
generated by $A$ is denoted by $\langle A \rangle$,
and it is defined as the intersection of all ideals 
of $S$ containing $A$. 
Note that $\langle A \rangle = \Z A + AS + SA + SAS$.

For the rest of this section, $S$ is equipped with a 
$G$-grading.
We let $h(S)$ denote the set $\cup_{g \in G} S_g$
of homogeneous elements of $S$.
For a nonzero $s\in h(S)$, we put $\deg(s):=g$ if $s \in S_g$.

\begin{prop}\label{prop:gradedidealgenerated}
Let $A$ be a subset of $h(S)$. Then 
$\langle A \rangle$ is a graded ideal of $S$.
\end{prop}

\begin{proof}
Every element of 
$\langle A \rangle$ is a finite sum of homogeneous 
elements of the 
form $a$, $sa$, $at$ or $s a t$, for $a \in A \subseteq h(S)$ and $s,t \in h(S)$. Thus, $\langle A \rangle$ is graded.
\end{proof}

We say that a subset $T$ of a graded ring $T'$ is 
a \emph{graded $m$-system in $T'$}, if for all $a,b \in T \cap h(T')$, 
we have $ab\in T$ or there is some $s \in h(T')$ such that $asb \in T$. 
	
\begin{prop}\label{prop:gradedmsystem}
Let $P$ be a proper graded ideal of $S$.
The following assertions are equivalent:
\begin{enumerate}[{\rm (i)}]

\item The ideal $P$ is graded prime.

\item For all $a,b \in h(S)$, if $a S b \subseteq P$ and $ab \in P$,
then $a \in P$ or $b \in P$. 

\item The complement $S \setminus P$ is
a graded $m$-system in $S$.

\end{enumerate}
\end{prop}

\begin{proof}
(i)$\Rightarrow$(ii):
Suppose that $P$ is graded prime and 
$a,b \in h(S)$ satisfy $a S b \subseteq P$ and $ab\in P$. By 
Proposition~\ref{prop:gradedidealgenerated}, 
$\langle a \rangle$
and $\langle b \rangle$ are graded ideals of $S$. 
We get that
\begin{align*}
\langle a \rangle \langle b \rangle &= 
(\Z a + aS + Sa + SaS)
(\Z b + bS + Sb + SbS) \\
&=\Z a b + \Z a bS + \Z a Sb + \Z a SbS
+ \Z aS b + aS bS + aS Sb + aS SbS \\
&+ \Z Sa b + Sa bS + Sa Sb + Sa SbS
+ \Z SaS b + SaS bS + SaS Sb + SaS SbS \subseteq P.
\end{align*}

Since $P$ is graded prime it follows that
$\langle a \rangle \subseteq P$ or $\langle b \rangle 
\subseteq P$ which, in turn, implies that
$a \in P$ or $b \in P$.

(ii)$\Rightarrow$(iii):
Suppose that (ii) holds. Take $a,b \in h(S) \setminus P$.
From (ii) it follows that $ab \notin P$ or there is 
$s \in S$ with $a s b \notin P$.
In the latter case, there is an 
$s' \in h(S)$ such that $a s' b \notin P$. 

(iii)$\Rightarrow$(i):
Suppose that $S \setminus P$ is
a graded $m$-system in $S$. 
Let $A$ and $B$ be graded ideals of $S$
with $AB \subseteq P$. Seeking a contradiction,
suppose that $A \setminus P \neq \emptyset$ and 
$B \setminus P \neq \emptyset$. 
Every element in $A$ and $B$ 
is a finite sum of homogeneous elements in 
$A$ and $B$, respectively,
and hence there exist
$a \in h(A) \setminus P$ and 
$b \in h(B) \setminus P$. 
Since $S \setminus P$ is a graded $m$-system in $S$, 
we have $ab\notin P$ which is a contradiction, or that
there is an $s \in h(S)$ with $a s b \notin P$.
But then we get the contradiction
$S \setminus P \ni asb \in aSb \subseteq
\langle a \rangle \langle b \rangle
\subseteq AB \subseteq P$.
Therefore $A \subseteq P$ or $B \subseteq P$.
\end{proof}

If we restrict ourselves to trivial $G$-gradings on a unital ring $S$,
that is gradings where $S_g = \{ 0 \}$,
for $g \neq e$, then a graded $m$-system coincides with the
classical notion of an $m$-system (see \cite[Def. 10.3]{lam2001}).	From Proposition~\ref{prop:gradedmsystem}, we
therefore immediately get the following:

\begin{cor}\label{cor:msystem}
Let $P$ be a proper ideal of $S$.
The following assertions are equivalent:
\begin{enumerate}[{\rm (i)}]
\item The ideal $P$ is prime.
\item For all $a,b \in S$, if $a S b \subseteq P$ and $ab \in P$, then $a \in P$ or $b \in P$. 
\item The complement $S \setminus P$ is an $m$-system in $S$.
\end{enumerate}
\end{cor}

\begin{thm}\label{thm:Pgradedprime}
Suppose that $G$ is an ordered group and that $S$ is a $G$-graded ring.
Let $P$ be a proper graded ideal of $S$.
Then $P$ is prime if and only if $P$ is graded prime.
\end{thm}

\begin{proof}
The ``only if'' statement is trivial. Now we show the ``if''
statement. Suppose that $P$ is graded prime.
We claim that $S \setminus P$ is an $m$-system in $S$.
If we assume that the claim holds, then, from 
Corollary \ref{cor:msystem}, it follows that $P$ is prime.
Now we show the claim. Take $a,b \in S \setminus P$.
Let $p$ (resp. $q$) be the sum of the homogeneous components 
of $a$ (resp. $b$) that belong to $P$.
Put $d := a-p$ and $d' := b-q$. Note that all components
of $d$ and $d'$ belong to $S \setminus P$.
Since $G$ is ordered we can pick the nonzero components
$r$ and $r'$ of $d$ and $d'$, respectively, of highest degree.
Since $P$ is graded prime, it follows from
Proposition \ref{prop:gradedmsystem} that
$S \setminus P$ is a graded $m$-system in $S$.
Thus, there exists a nonzero $s \in h(S)$ such that 
$r s r' \notin P$ or $rr' \notin P$. 
Seeking a contradiction, suppose that $asb \in P$ and $ab\in P$.
From the equality
$a s b = (d + p)s(d' + q) = 
dsd' + p s d' + dsq + psq$
and the fact that $p s d',dsq,psq \in P$, 
it follows that $dsd' \in P$.
Similarly, $ab \in P$ leads to $dd' \in P$.
Since $P$ is graded, all components of $dsd'$ (resp. $dd'$) also
belong to $P$. However, the component of $dsd'$ (resp. $dd'$)
of highest degree is $r s r'$ (resp. $rr'$) which belongs to $S \setminus P$. This is a contradiction.
\end{proof}

Theorem~\ref{thm:Pgradedprime} immediately
implies:

\begin{thm}\label{thm:Gorderedprimegradedprime}
Suppose that $G$ is an ordered group and that $S$ is a $G$-graded ring. Then $S$ is prime
if and only if $S$ is graded prime.
\end{thm}

Note that 
Theorem \ref{thm:Gorderedprimegradedprime} 
is a partial generalization of 
\cite[Thm. 3.2]{abrams1993}.

\section{$G$-prime ideals}\label{sec:Gprimeideals}

Throughout this section, $S$ denotes an associative $G$-graded ring that is not
necessarily unital.
Let $J$ be an ideal of $S_e$. Recall 
from Section \ref{sec:introduction} that
$J$ is said to be $G$-invariant if
for each $x \in G$, the inclusion
$S_{x^{-1}} J S_x \subseteq J$ holds.
We define the map $( \cdot )_e : S \to S_e$, by $(\sum_{g\in G} s_g )_e := s_e$.
By abuse of notation, the parentheses in $(\cdot)_e$ will often be suppressed.

\begin{prop}\label{prop:JeJ}
Let $J$ be an ideal of $S_e$. Then $J$ is 
$G$-invariant if and only if $\langle J \rangle_e = J$.
\end{prop}

\begin{proof}
Since $J \subseteq \langle J \rangle_e
= (J + JS + SJ + SJS)_e = 
J + JS_e + S_e J + 
\sum_{x \in G} S_{x^{-1}} J S_x = 
J + \sum_{x \in G} S_{x^{-1}} J S_x$,
we get that $J \subseteq \langle J \rangle_e = 
J + \sum_{x \in G} S_{x^{-1}} J S_x$,
from which it easily follows that
$J$ is $G$-invariant
if and only if $\langle J \rangle_e = J$.
\end{proof}

\begin{prop}\label{prop:Iegraded}
Let $I$ be a graded ideal of $S$. Then $I_e$ is a 
$G$-invariant ideal of $S_e$.
\end{prop}

\begin{proof}
Take $x \in G$. Then $S_{x^{-1}} I_e S_x \subseteq
(S_{x^{-1}} I S_x) \cap (S_{x^{-1}} S_e S_x)
\subseteq I \cap S_e = I_e$.
\end{proof}

\begin{defi}
Let $I$ be a graded ideal of $S$.  
We say that $I$ is 
\emph{identity generated} if $I = \langle I_e \rangle$.
\end{defi}

\begin{prop}\label{prop:bijection1}
Let $S$ be a $G$-graded ring.
The map $\langle \cdot \rangle$ is an inclusion
preserving bijection 
$\{ \mbox{$G$-invariant ideals of $S_e$} \}
\to 
\{ \mbox{identity generated $G$-graded ideals of $S$} \}$
with inverse $(\cdot)_e$.
\end{prop}

\begin{proof}
Clearly the maps $\langle \cdot \rangle$ and 
$(\cdot)_e$ are inclusion preserving.
By Propositions~\ref{prop:gradedidealgenerated} and \ref{prop:JeJ}, it follows that
$\langle \cdot \rangle$ is well defined and 
$(\cdot)_e \circ \langle \cdot \rangle = {\rm id}$.
Proposition \ref{prop:Iegraded} implies that
$(\cdot)_e$ is well defined.
From the definition of identity generated graded ideals 
it follows that $\langle \cdot \rangle \circ (\cdot)_e
= {\rm id}$.
\end{proof}

Recall that $S$ is said to be \emph{symmetrically $G$-graded} if $S_x S_{x^{-1}} S_x = S_x$ for every $x \in G$.

\begin{defi}\label{def:ideallysymmetricallygraded}
We say that $S$ is \emph{ideally symmetrically $G$-graded},
if for every graded ideal $I$ of $S$ and 
every $x \in G$
the equalities $S_x S_{x^{-1}} I_x = I_x$ and $I_x S_{x^{-1}} S_x = I_x$ hold.
\end{defi}

\begin{prop}\label{prop:leftideallysymm}
The following assertions hold:
\begin{enumerate}[{\rm (a)}]

\item If $S$ is ideally symmetrically $G$-graded,
then $S$ is symmetrically $G$-graded.

\item If $S$ is nearly epsilon-strongly $G$-graded,
then $S$ is ideally symmetrically $G$-graded.

\end{enumerate}
\end{prop}

\begin{proof}
(a) Take $I = S$ in 
Definition~\ref{def:ideallysymmetricallygraded}.

(b) Suppose that $S$ is nearly epsilon-strongly $G$-graded. Take $x \in G$ and a graded ideal $I$ of $S$.
Then $S_x S_{x^{-1}} I_x \subseteq S_e I_x \subseteq I_x$ and $I_x S_{x^{-1}} S_x \subseteq I_x S_e \subseteq I_x$.
Now we show the reversed inclusions. Take $s \in I_x$. By the definition of nearly
epsilon-strongly $G$-graded rings 
(see Section \ref{sec:introduction}),
there are $\epsilon_x \in S_x S_{x^{-1}}$
and $\epsilon_x' \in S_{x^{-1}} S_x$ such that
$s = \epsilon_x s = s \epsilon_x'$.
Thus, $s = \epsilon_x s \in S_x S_{x^{-1}} I_x$ 
and
$s = s \epsilon_x' \in I_x S_{x^{-1}} S_x$. This shows that $S_x S_{x^{-1}}I_x=I_x=I_x S_{x^{-1}}S_x$.
\end{proof}

\begin{rem}\label{rem:symmbutnotsymm}
There exist rings which are symmetrically graded but not ideally symmetrically graded (see Proposition \ref{prop:example}).
\end{rem}

\begin{prop}\label{prop:ideallygraded}
Suppose that $S$ is ideally symmetrically $G$-graded.
Let $I$ be a graded 
ideal of $S$. 
Then $I = S I_e = \langle I_e \rangle$
and
$I = I_e S = \langle I_e \rangle$.
\end{prop}

\begin{proof}
Note 
that 
$I = \sum_{x \in G} I_x = 
\sum_{x \in G} S_x S_{x^{-1}} I_x \subseteq 
\sum_{x \in G} S_x I_e \subseteq S I_e  \subseteq SI \subseteq I$
since $S_{x^{-1}} I_x \subseteq S_e \cap I = I_e$ 
for $x \in G$. 
Thus, 
$I = S I_e \subseteq I_e + S I_e + I_e S + S I_e S = \langle I_e \rangle \subseteq I$
which establishes the first set of equalities. The second set of equalities are shown analogously.
\end{proof}

\begin{prop}\label{prop:leftright}
Let $S$ be  
ideally symmetrically $G$-graded.
Then $\langle \cdot \rangle$ is an inclusion
preserving bijection 
$\{ \mbox{$G$-invariant ideals of $S_e$} \} 
\to
\{ \mbox{$G$-graded ideals of $S$} \}$
with inverse $(\cdot)_e$.
\end{prop}

\begin{proof}
This follows from 
Proposition \ref{prop:bijection1}
and Proposition \ref{prop:ideallygraded}.
\end{proof}

\begin{prop}\label{prop:primeA}
Suppose that $S$ is  
ideally symmetrically $G$-graded.
Let $I$ be a graded ideal of $S$ such that 
$I_e$ is a $G$-prime ideal of $S_e$. 
Then $I$ is graded prime.
\end{prop}

\begin{proof}
Let $A$ and $B$ be graded ideals of $S$ with
$AB \subseteq I$. Then $A_e B_e \subseteq AB \cap S_e
\subseteq I \cap S_e = I_e$. By Proposition \ref{prop:Iegraded},
$A_e$ and $B_e$ are $G$-invariant.
Since $I_e$ is $G$-prime, we have $A_e \subseteq I_e$
or $B_e \subseteq I_e$. By Proposition \ref{prop:ideallygraded}
and Proposition \ref{prop:leftright},
$A = \langle A_e \rangle \subseteq 
\langle I_e \rangle = I$ or 
$B = \langle B_e \rangle \subseteq \langle I_e \rangle = I$. 
\end{proof}

\begin{prop}\label{prop:primeB}
Suppose that $S$ is  
ideally symmetrically $G$-graded.
Let $I$ be an ideal of $S$ which is graded prime. 
Then $I_e$ is $G$-prime.
\end{prop}

\begin{proof}
Let $J$ and $K$ be $G$-invariant ideals of $S_e$
with $JK \subseteq I_e$. From Propositions~\ref{prop:gradedidealgenerated}, \ref{prop:JeJ} and \ref{prop:ideallygraded}, it follows that
$\langle J \rangle \langle K \rangle = 
SJ KS \subseteq \langle JK \rangle \subseteq \langle I_e \rangle = I.$
By graded primeness of $I$, we have 
$\langle J \rangle \subseteq I$ or
$\langle K \rangle \subseteq I$.
Proposition \ref{prop:leftright} now implies that
$J = \langle J \rangle_e \subseteq I_e$ or
$K = \langle K \rangle_e \subseteq I_e$.
Thus $I_e$ is $G$-prime.
\end{proof}

\begin{prop}\label{prop:leftrightprime}
Let $S$ be 
ideally symmetrically $G$-graded.
Then $\langle \cdot \rangle$ is an inclusion
preserving bijection
$\{ \mbox{$G$-prime ideals of $S_e$} \} 
\to
\{ \mbox{$G$-graded prime ideals of $S$} \}$
with inverse $(\cdot)_e$.
\end{prop}

\begin{proof}
This follows from Propositions \ref{prop:leftright},
\ref{prop:primeA} and \ref{prop:primeB}.
\end{proof}

\begin{thm}\label{thm:GorderedGprime}
Let $G$ be an ordered group.
Suppose that $S$ is  
ideally symmetrically $G$-graded.
Then $S$ is prime
if and only if $S_e$ is $G$-prime.
\end{thm}

\begin{proof}
This follows from Theorem \ref{thm:Gorderedprimegradedprime}
and Proposition \ref{prop:leftrightprime}.
\end{proof}

\section{Application: prime Leavitt 
path rings}\label{sec:LPA}

In \cite{lannstrom2025} a characterization of prime Leavitt path rings was obtained as an application of some rather involved methods. Here, we give a shorter proof of the same result (see Theorem~\ref{thm:shortLPA})
by applying Theorem \ref{thm:GorderedGprime}.

Let $E = (E^0, E^1, s, r)$ be a \emph{directed graph},
where $E^0$ is the set of vertices,
$E^1$ is the set of edges, and 
$s,r \colon E^1 \to E^0$ are the source and range maps, respectively. 
For any $v \in E^0$, define 
$s^{-1}(v):=\{ e \in E^1 \mid s(e)=v \}$, 
the set of edges emitted from $v$.  
A \emph{path} in $E$ is a sequence 
$\alpha := f_1 f_2 \dots f_n$ where $f_i \in E^1$ and $r(f_{i}) = s(f_{i+1})$, for all $i$, and such a path is said to have {\it length} $n$. 
Vertices are considered to be  paths of length zero. 
The set of all paths is denoted $E^*$. 
Define a preorder on $E^0$ by writing
$u \geq v$ if there is a path from $u$ to $v$. 
The graph $E$ is said to satisfy 
\emph{condition~(MT-3)} if for all
$u, v \in E^0$, there is 
$w \in E^0$ with $u \geq w$ and $v \geq w$.

For the rest of this section, $R$ denotes an associative unital ring. 
Recall that an \emph{$R$-ring $S$} is a ring that is also an $R$-bimodule such that the following three conditions are satisfied for all $r \in R$ and $s, s' \in S$:
\begin{itemize}
\item $(ss')r = s(s'r)$;
\item $(rs)s' = r(ss')$;
\item $(sr)s' = s(rs')$.
\end{itemize}

The \emph{Leavitt path ring of $E$ over $R$}, denoted by $L_R(E)$, is the associative $R$-ring generated by $\{ v \mid v \in E^0 \} \cup \{ f \mid  f \in E^1 \} \cup \{ f^* \mid f \in E^1 \}$ subject to the following six families of relations:
\begin{enumerate}[{\rm (i)}]

\item $v w = \delta_{v,w} v$ for all $v, w \in E^0$; 

\item $s(f) f = f r(f)=f$  for every $f \in E^1$;

\item $r(f)f^* = f^*s(f)=f^*$  for every $f \in E^1$;

\item $f^* f' = \delta_{f, f'} r(f)$ for all $f, f' \in E^1$;

\item $\sum_{f \in E^1, s(f)=v}f f^* =v $ for every $v \in E^0$ with $s^{-1}(v)$ nonempty and finite; 

\item $rv=vr$, $rf=fr$ and $rf^*=f^*r$, for all $r\in R$ and $f\in E^1$.
\end{enumerate}

\begin{rem}\label{rem:nonzero}
If $\alpha \in E^*$, then $\alpha$ and $\alpha^*$ are referred to as a \emph{real path} and a \emph{ghost path}, respectively, in $L_R(E)$.
By \cite[Lemma 3.5]{lundstrom2023} 
every element in  $E^0 \cup E^1 \cup \{ f^* \mid f \in E^1\}$ 
is nonzero in $L_R(E)$, and the set of real
(respectively ghost) paths is linearly 
independent in the left $R$-module 
$L_R(E)$ and in the right $R$-module 
$L_R(E)$. From this it follows that
all elements of the form $r \alpha \beta^*$
for nonzero $r \in R$ and 
$\alpha,\beta \in E^*$ with 
$r(\alpha)=r(\beta)$, are nonzero.
The element $\alpha\beta^*$ is referred to as a \emph{monomial} in $L_R(E)$.
\end{rem}

\begin{rem}\label{rem:canonicalgrading}
The ring $L_R(E)$ carries a canonical $\Z$-grading defined by ${\rm deg}(f):=1$ and ${\rm deg}(f^*):=-1$ for $f \in E^1$, and ${\rm deg}(v):=0$, for $v \in E^0$ (see \cite[Cor.~2.1.5]{abrams2017}).
This grading is 
nearly epsilon-strong (see \cite{nystedtoinert2019}),
and thus, by 
Proposition~\ref{prop:leftideallysymm}(b),
$L_R(E)$ is 
ideally symmetrically $\Z$-graded.
\end{rem}

\begin{lem}\label{lem:Znonzero}
Suppose that $E$ is a directed
graph and $R$ is a unital ring.
Let $I$ be a nonzero $\mathbb{Z}$-invariant 
ideal of $L_R(E)_0$. Then there is
$v \in E^0$ and a nonzero $r \in R$
such that $rv \in I$.
\end{lem}

\begin{proof}
Put $S := L_R(E)$.
Take a nonzero $a \in I$.
By \cite[Prop. 14.11]{lannstrom2025} there exist 
$\alpha,\beta \in E^*$, $v \in E^0$ and $r \in R \setminus \{ 0 \}$ with
$\alpha^* a \beta = rv$. 
Since $rv \in S_0$ it follows from 
the $\mathbb{Z}$-grading of $S$ 
that $\deg(\alpha)=\deg(\beta) =: n$.
Since $I$ is $\mathbb{Z}$-invariant, 
$rv = \alpha^* a \beta \in S_{-n} I S_n
\subseteq I$.
\end{proof}

\begin{thm}\label{thm:shortLPA}
Suppose that $E$ is a directed graph and 
$R$ is a unital ring.
Then the Leavitt path ring $L_R(E)$ is
prime if and only if $R$ is prime and 
$E$ satisfies condition (MT-3).
\end{thm}

\begin{proof}
Put $S := L_R(E)$. 
First we show the
``only if'' statement.
Suppose that $R$ is not prime.
Take nonzero ideals $I,J$ of $R$
with $IJ = \{ 0 \}$. 
Let $A$ and $B$ be the ideals of $S$ 
consisting of sums of monomials
with coefficients in $I$ and $J$,
respectively. By Remark \ref{rem:nonzero},
$A$ and $B$ are nonzero.
Clearly, $AB = \{ 0 \}$, and hence 
$S$ is not prime.
Suppose now that $E^0$ does not 
satisfy condition (MT-3).
Then there exist $v, w \in E^0$ such
that for every $y \in E^0$,
$v \ngeq  y$ or $w \ngeq y$.
Put $C := SvS$ and
$D := SwS$ of $S$. 
By Remark \ref{rem:nonzero},
$C$ and $D$ are
nonzero ideals of $S$. 
Seeking a contradiction,
suppose that $CD \neq \{ 0 \}$. Then there
exist $\alpha,\beta \in E^*$ with 
$v \alpha \beta^* w \neq 0$ which gives
the contradiction $v \geq u$ and
$w \geq u$ with $u := r(\alpha)=r(\beta)$.
Therefore $CD = \{ 0 \}$. Hence $S$ is
not prime.

Now we show the ``if'' statement.
Suppose that $R$ is prime and 
$E$ satisfies (MT-3).
Seeking a contradiction,
suppose that $S$ is not prime.
By Theorem \ref{thm:GorderedGprime} and 
Remark \ref{rem:canonicalgrading},
there exist nonzero $\mathbb{Z}$-invariant
ideals $A$ and $B$ of $S_0$ with
$AB = \{ 0 \}$. By Lemma \ref{lem:Znonzero},
there are $r,s \in R \setminus \{ 0 \}$ and 
$v,w \in E^0$ with $rv \in A$ and 
$sw \in B$. From (MT-3) there are 
$\alpha,\beta \in E^*$ and $u \in E^0$ with 
$s(\alpha)=v$, $r(\alpha) = r(\beta)= u$
and $s(\beta)=w$. 
Consider the nonzero ideals $I := RrR$ and 
$J := RsR$ of $R$. By primeness of $R$ and
Remark \ref{rem:nonzero} it follows that 
$(IJ)u \neq \{ 0 \}$. Since 
$\alpha^* \alpha = u = \beta^* \beta$,
$\mathbb{Z}$-invariance of $A$ and $B$
imply that $Iu = \alpha^*(I v)\alpha
\subseteq \alpha^* A \alpha
\subseteq A$ and $Ju = \beta^*(Jw) \beta
\subseteq \beta^* B \beta
\subseteq B$. This gives the sought 
contradiction:
$\{ 0 \} = AB \supseteq (Iu)(Ju) = 
(IJ)u \neq \{ 0 \}$.
\end{proof}

\section{Application: prime graded subrings of group rings}\label{sec:primegradedsubringsofgrouprings}

The purpose of this section is twofold.
Given any nontrivial group $G$, we will show that
\begin{itemize}

\item the class of ideally symmetrically $G$-graded graded 
rings strictly contains the class of epsilon-strongly
$G$-graded rings (see Proposition~\ref{prop:RL}), and

\item the class of symmetrically $G$-graded rings
strictly contains the class of  
ideally symmetrically $G$-graded graded 
rings (see Proposition~\ref{prop:example}).

\end{itemize}
To establish this, we consider certain
subrings of group rings defined by ideal filters.
Namely, throughout this section, 
we make the following assumptions:
\begin{itemize}
\item  $R$ is an associative 
but not necessarily unital ring, and $G$ is a group;
\item 
for each $x \in G$,
we are given an additive subgroup 
$I_x$ of $R$ with $I_e = R$;

\item we consider the additive subgroup 
$S := \oplus_{x \in G} I_x x$
of the group ring $R[G]$ of $G$ over $R$;

\item for each $x \in G$, put 
$S_x := I_x x$, so that the additive 
group $S$ is $G$-graded
by $(S_x)_{x \in G}$.

\end{itemize}
We say that $(I_x)_{x \in G}$ is a 
\emph{$G$-filter} if for all $x,y \in G$,
the inclusion 
$I_x I_y \subseteq I_{xy}$ holds.

\begin{prop}\label{prop:gradedsubring}
The additive $G$-graded group $S$ is a 
$G$-graded subring of $R[G]$ 
if and only if $(I_x)_{x \in G}$ is a 
$G$-filter of ideals of $R$.
\end{prop}

\begin{proof}
Suppose that $S$ is a $G$-graded subring of $R[G]$.
Take $x,y \in G$.
Since $I_x$ is an additive subgroup of 
$S$ and the 
inclusions $S_e S_x \subseteq S_x$ and 
$S_x S_e \subseteq S_x$ hold, 
it follows that
$I_x$ is an ideal of $R$. The inclusion
$S_x S_y \subseteq S_{xy}$ implies that 
$I_x I_y \subseteq I_{xy}$. 
This establishes the
``only if'' part of the statement. 
The ``if''
part of the statement follows analogously.
\end{proof}

\begin{prop}\label{prop:symmetricidempotent}
Suppose that $S$ is a $G$-graded subring of $R[G]$.
Then $S$ is symmetrically $G$-graded if and only if 
$R$ is idempotent and for each $x \in G$, the equality 
$I_x I_{x^{-1}} I_x = I_x$ holds.
\end{prop}

\begin{proof}
The ring $S$ is symmetrically $G$-graded if and only 
if the equality $S_x S_{x^{-1}} S_x = S_x$
holds for every $x \in G$. This is, in turn, equivalent to 
$I_x I_{x^{-1}} I_x = I_x$ for every $x \in G$.
Putting $x = e$ in this equality yields
$R R R = R$ so that $R = R R R \subseteq R R \subseteq R$. Hence, $R = RR$.
\end{proof}

Recall that a ring $R$ is called \emph{fully idempotent}
if every ideal of $R$ is idempotent. In that case,
in particular, $R$ itself has to be idempotent.
From \cite[Thm. 1.2]{courter1969} we recall the
following:

\begin{prop}\label{prop:Rfullyidempotent}
A ring is fully idempotent
if and only if for all ideals $I$ and $J$ 
of the ring the
equality $IJ = I \cap J$ holds.
\end{prop}

\begin{prop}\label{prop:gginverse}
Let $S$ be a $G$-graded subring of $R[G]$
where $R$ is fully idempotent.
Then $S$ is symmetrically $G$-graded if and only if 
for every $x \in G$ the equality $I_x = I_{x^{-1}}$ holds.
\end{prop}

\begin{proof}
By Propositions \ref{prop:symmetricidempotent} and
\ref{prop:Rfullyidempotent}, 
$S$ is symmetrically $G$-graded
$\Leftrightarrow$
$\forall x \in G$: 
$I_x \cap I_{x^{-1}} = I_x$
$\Leftrightarrow$
$\forall x \in G$: 
$I_x \subseteq I_{x^{-1}}$ 
$\Leftrightarrow$ 
$\forall x \in G$: $I_x = I_{x^{-1}}$.
\end{proof}

\begin{prop}\label{prop:SsymmSrightsymm}
Let $S$ be a $G$-graded subring of $R[G]$ where $R$ is fully idempotent.
Then $S$ is symmetrically $G$-graded if and only if
$S$ is ideally symmetrically $G$-graded.
\end{prop}

\begin{proof}
The ``if'' statement follows from Proposition~\ref{prop:leftideallysymm}(a).
Now we show the ``only if'' statement.
Suppose that $S$ is symmetrically $G$-graded.
Let $J$ be a graded ideal of $S$. 
Take $x \in G$.
Since $RJ \subseteq J$
and $JR \subseteq J$ it follows that 
there is an ideal $K_x$ of $R$ such that $J_x = K_x x$.
From $J \subseteq S$ we get that $K_x \subseteq I_x$.
By Proposition~\ref{prop:Rfullyidempotent}
and Proposition~\ref{prop:gginverse}, it follows that
$S_x S_{x^{-1}} J_x = 
I_x I_{x^{-1}} K_x x = 
I_x I_x K_x x = 
(I_x \cap I_x \cap K_x)x =
K_x x = J_x$.
Analogously, one can show that $J_x  S_{x^{-1}} S_x = J_x$.
\end{proof}

\begin{thm}\label{thm:SsymmRfull}
Suppose that $S$ is a symmetrically $G$-graded
subring of $R[G]$ where the group $G$ is ordered 
and the ring $R$ is fully idempotent.
Then $S$ is prime if and only if $R$ is prime. 
\end{thm}

\begin{proof}
This follows from Theorem \ref{thm:GorderedGprime} and Proposition \ref{prop:SsymmSrightsymm}, since ideals of $R$ are $G$-invariant.
\end{proof}

We now proceed to show that there exist symmetrically graded subrings of \( R[G] \) that are not nearly epsilon-strongly graded (see Proposition~\ref{prop:RL}). For this, we recall some definitions. Let \( T \) be an associative (not necessarily unital) ring and \( M \) a left (resp. right) \( T \)-module. Then \( M \) is \emph{$s$-unital} if \( m \in Tm \) (resp. \( m \in mT \)) for all \( m \in M \). The ring \( T \) is $s$-unital if it is $s$-unital as both a left and right \( T \)-module. From \cite[Cor.~1.5]{courter1974}, we recall the following:

\begin{prop}
Let $R$ be a fully idempotent ring.
The center of $R$ is von Neumann regular. 
In particular, if $R$ is commutative, then 
$R$ is von Neumann regular and hence $s$-unital.
\end{prop}

Thus, to find a non-$s$-unital fully idempotent \( R \), we must consider a noncommutative $R$. We also aim to consider rings \( S \) that are not nearly epsilon-strongly \( G \)-graded. 

\begin{prop}\label{prop:nearlynearly}
Suppose that $S$ is a $G$-graded subring of $R[G]$.
Then $S$ is nearly epsilon-strongly $G$-graded
if and only if for every $x \in G$ the ideal 
$I_x$ is $s$-unital both as a left 
$I_x I_{x^{-1}}$-module 
and as a right $I_{x^{-1}} I_x$-module.
\end{prop}

\begin{proof}
The claim follows from the chain of equivalences:
$S$ is nearly epsilon-strongly $G$-graded $\Leftrightarrow$
$\forall x \in G$, $S_x$ is $s$-unital both as a left
$S_x S_{x^{-1}}$-module and as a right
$S_{x^{-1}} S_x$ module $\Leftrightarrow$
$\forall x \in G$, $I_x$ is $s$-unital both as a left
$I_x I_{x^{-1}}$-module and as a right
$I_{x^{-1}} I_x$ module. 
\end{proof}

\begin{cor}\label{cor:notepsilon}
Let $S$ be a symmetrically $G$-graded
subring of $R[G]$ where
$R$ is a fully idempotent unital domain. 
Suppose that there is $x \in G$ with
$0 \subsetneq I_x \subsetneq R$.
Then $S$ is not nearly epsilon-strongly $G$-graded.
\end{cor}

\begin{proof}
Seeking a contradiction, suppose that
$S$ is nearly epsilon strongly $G$-graded.
Take a nonzero $a \in I_x$.
From Proposition \ref{prop:Rfullyidempotent} and 
Proposition \ref{prop:gginverse} it follows that
$I_x I_{x^{-1}} = I_x I_x = I_x$.
By Proposition~\ref{prop:nearlynearly},
there is $b \in I_x I_{x^{-1}} = I_x$ with 
$ba = a$. But $R$ is unital and hence $1_R a = a$,
so that $(b-1_R)a = ba - 1_R a = a - a = 0$.
Since $a \neq 0$ and $R$ is a domain this implies that
$b = 1_R$ which contradicts $I_x \subsetneq R$.
\end{proof}

\begin{exa}\label{ex:RL}
We follow \cite[Ex. 12.2]{lam1995} closely.
Let $F_0$ be a field of characteristic zero.
Let $A$ be the first Weyl algebra over $F_0$
with generators $X$ and $Y$ satisfying the relation
$XY - YX = 1$. Let $\delta$ be the differentiation 
operator on $F := F_0[Y]$. Then $A$ is the differential 
polynomial ring $F[X;\delta]$. Put $L := AX$.
Then $L$ is a simple domain without identity.
Put $R := L \oplus F_0$. Then $R$ is a fully
idempotent domain with exactly three ideals, 
$\{ 0 \}$, $L$ and $R$.
\end{exa}

\begin{prop}\label{prop:RL}
Let $R$ and $L$ be defined as in 
Example \ref{ex:RL}.
Suppose that we for each $x \in G$ 
take $I_x \in \{ R,L \}$
such that the set 
$G_R := \{ x \in G \mid I_x = R \}$
is a subgroup of $G$. The following assertions hold:
\begin{enumerate}[{\rm (a)}]

\item The ring $S$ is symmetrically $G$-graded.

\item The ring $S$ is strongly $G$-graded $\Leftrightarrow$
$S$ is nearly epsilon-strongly $G$-graded 
$\Leftrightarrow$ $G_R = G$.

\item If $G$ is an ordered group, then $S$ is prime.

\end{enumerate}
\end{prop}

\begin{proof}
(a) Since $G_R$ is a subgroup of $G$,
Proposition \ref{prop:gradedsubring} and
Proposition \ref{prop:gginverse} imply that 
$S$ is a symmetrically $G$-graded subring of $R[G]$.

(b) Since $R$, and hence also $S$, is unital,
the implication $S$ is strongly $G$-graded $\Rightarrow$
$S$ is nearly epsilon-strongly $G$-graded holds.
From Corollary \ref{cor:notepsilon} the implication
$S$ is nearly epsilon-strongly graded 
$\Rightarrow$ $G_R = G$ holds.
The implication $G_R = G$ $\Rightarrow$
$S$ is strongly $G$-graded, is trivial.

(c) This follows from Theorem \ref{thm:SsymmRfull},
since $R$ is a domain and hence prime.
\end{proof}

We end this article by providing an example of the claim made in Remark~\ref{rem:symmbutnotsymm} (see Proposition~\ref{prop:example}). To this end, we consider the following:

\begin{exa}\label{ex:JK}
Let $\overline{\Q}$ 
denote the field of
algebraic numbers, that is the algebraic
closure of $\Q$. Let $O$ denote 
the ring of algebraic integers in 
$\overline{\Q}$.
Let $J$ be the ideal of $O$ generated
by $\{ 2^{1/2^n} \mid n \geq 1 \}$. 
Since $2^{1/2^n} = ( 2^{1/2^{n+1}} )^2$
for $n \geq 1$ it follows that $J$
is an idempotent ideal. 
We claim that $1 \notin J$.
Assume, for a moment, that the claim holds.
Then $J$ is a proper
ideal of $O$.
Put $K := O \sqrt{2}$.
Then $K$ is an ideal of the ring $J$.
Since $1 \notin J$ it follows that
$\sqrt{2} \notin J \sqrt{2}$ and hence
that $K = O \sqrt{2} \supsetneq 
J O \sqrt{2} = JK$.
Note that $K \subsetneq J$,
since supposing that $J = K$, then
using idempotency of $J$, would give 
$J = J^2 = JK \subsetneq J$.
Summarizing, we have 
$JK \subsetneq K \subsetneq J$.

Now we show the claim. 
Seeking a contradiction, suppose that $1 \in J$.
Then there exist $m\ge1$, $a_1,\dots,a_m\in O$,
and integers $n_1,\dots,n_m\ge1$ with
$1 = \sum_{i=1}^m a_i\,2^{1/2^{n_i}}$.
Let $N:=\max\{n_1,\dots,n_m\}$. 
For each $i$, we have
$2^{1/2^{n_i}} \;=\; \bigl(2^{1/2^{N}}\bigr)^{\,2^{\,N-n_i}}$,
so each summand $a_i\,2^{1/2^{n_i}}$ lies in the principal ideal $\bigl(2^{1/2^{N}}\bigr)$. Therefore $1 \in \bigl(2^{1/2^{N}}\bigr)$ and so
$1 = 2^{1/2^{N}} \sum_{i=1}^m a_i\,\bigl(2^{1/2^{N}}\bigr)^{\,2^{\,N-n_i}-1}$.
Since $O$ is a ring and $2^{1/2^N}$ is an algebraic integer, the sum belongs to $O$. Hence $2^{1/2^{N}}$ is a unit in $O$,
which is a contradiction.
Indeed, if $2^{1/2^N}$ were a unit in $O$, then
$2=\bigl(2^{1/2^N}\bigr)^{2^N}$ would also be a 
unit in $O$, which is impossible since $1/2\notin O$.
\end{exa}

\begin{prop}\label{prop:example}
Let \( J \) and \( K \) be defined as in Example~\ref{ex:JK}, and let $I_x := J$ for every $x\in G$. Then \( S \) is symmetrically $G$-graded. Moreover, if $G$ is nontrivial, then
\( S \) is not ideally symmetrically $G$-graded.
\end{prop}

\begin{proof}
By Proposition~\ref{prop:symmetricidempotent},
$S$ is symmetrically $G$-graded, since $J$ is idempotent.
Now suppose that $G \neq \{ e \}$.
Put $M := Je + \sum_{e \neq x \in G} Kx$
and note that $M$ is a graded ideal of $S$.
Fix $x \in G \setminus \{ e\}$. Then
$S_x S_{x^{-1}} M_x = J^2 K x = JKx
\subsetneq Kx = M_x$. Thus, $S$ is not ideally symmetrically $G$-graded.
\end{proof}

\end{document}